\documentclass[10pt]{amsart}
\usepackage{amssymb}
\usepackage{graphicx}
\usepackage{float}
\usepackage[all,cmtip]{xy}
\newcommand{\s}{\mathfrak{s}}

\renewcommand{\O}{\mathcal{O}}

\newcommand{\D}{\mathbb{D}}

\renewcommand{\P}{\mathbb{P}}

\newcommand{\SA}{{\mathcal{A}}}

\newcommand{\SC}{{\mathcal{C}}}
\newcommand{\SD}{{\mathcal{D}}}

\newcommand{\SK}{{\mathcal{K}}}

\renewcommand{\SS}{{\mathcal{S}}}

\newcommand{\SU}{{\mathcal{U}}}

\newcommand{\J}{\mathfrak{j}}

\newcommand{\T}{\mathbb{T}}

\newcommand{\Z}{\mathbb{Z}}
\newcommand{\C}{\mathbb{C}}

\newcommand{\N}{\mathbb{N}}
\newcommand{\R}{\mathbb{R}}
\newcommand{\V}{\mathbb{V}}

\newcommand{\CP}{\mathbb{CP}}
\renewcommand{\S}{\mathbb{S}}
\renewcommand{\SC}{\mathcal{C}}
\renewcommand{\SA}{\mathcal{A}}
\renewcommand{\SD}{\mathcal{D}}

\newcommand{\Gr}{\operatorname{Gr}}

\newcommand{\im}{\operatorname{im}}

\newcommand{\End}{\operatorname{End}}

\newcommand{\Ham}{\operatorname{Ham}}
\newcommand{\codim}{\operatorname{codim}}
\newcommand{\Cont}{\operatorname{Cont}}
\newcommand{\Diff}{\operatorname{Diff}}
\newcommand{\Symp}{\operatorname{Symp}}

\newtheorem{proposition}{Proposition}
\newtheorem{theorem}[proposition]{Theorem}
\newtheorem{definition}[proposition]{Definition}
\newtheorem{lemma}[proposition]{Lemma}

\newtheorem{corollary}[proposition]{Corollary}

\theoremstyle{definition}
\newtheorem{remark}[proposition]{Remark}

\begin{document}

\title{Higher Maslov Indices}



\date{\today} 

\thanks{The work is partially supported by NSF grant DMS-1308501
  and by the Severo Ochoa Excellence Program}

\author{Roger Casals}
\address{Massachusetts Institute of Technology, Department of
  Mathematics, 77 Massachusetts Avenue Cambridge, MA 02139, USA}
\email{casals@mit.edu}

\author{Viktor Ginzburg}
\address{University of California Santa Cruz, California, United States}
\email{ginzburg@ucsc.edu}

\author{Francisco Presas}
\address{Instituto de Ciencias Matem\'aticas CSIC--UAM--UCM--UC3M,
C. Nicol\'as Cabrera, 13--15, 28049, Madrid, Spain}
\email{fpresas@icmat.es}



\begin{abstract}
  We define Maslov--type indices associated to contact and symplectic
  transformation groups. There are two such families of indices. The
  first class of indices are maps from the homotopy groups of the
  contactomorphism or symplectomorphism group to a quotient of
  $\Z$. These are based on a generalization of the Maslov index. The
  second class of indices are maps from the homotopy groups of the
  space of contact structures or the space of cohomologous symplectic
  forms to the homotopy groups of a simple homogeneous space. We
  provide a detailed construction and describe some properties of
  these indices and their applications.
\end{abstract}
\maketitle
\section{Introduction}\label{sec:intro}
In 1965, V.P. Maslov \cite[Chapitre II.2]{Ma} introduced an
integer--valued index to obtain asymptotic formulas for solutions
of partial differential equations with a small--parameter
$\hslash\to0$ in the leading term. This index allows one to extend the
Bohr--Sommerfeld quantization conditions appearing in the WKBJ method
beyond the caustic singularities, and also generalize the Morse index
for focal points in the calculus of variations. The resulting integer
index relations imply, at the first order, the quantization condition for
the quasi--classical solutions of the Schr\"odinger equation and, for
completely integrable systems, the Maslov quantization yields the exact
spectrum. As a consequence, the fiberwise Maslov quantization is
sufficient to quantize any integrable system \cite{Ar};
see also \cite{AB,DH}.

V.I. Arnol'd \cite{Ar} reinterpreted the Maslov index in terms of the
characteristic classes of a vector bundle and emphasized the role of
the Lagrangian submanifold given by the (evolution of the) initial
condition. The Maslov index is then obtained by counting the index of
the appropriate quadratic form associated to the projection onto the
configuration space. This index has become central to contact and
symplectic topology since the work of C.C. Conley and E. Zehnder
\cite{CZ} and its appearance in the Floer--type theories
\cite{Fl,EGH}. The aim of this article is to introduce two families of
indices, study their basic properties and provide some
applications. Their construction belongs to algebraic and symplectic
topology, and it is expected that applications analogous to those of
the Maslov index in asymptotic analysis and quasi--classical physics
can be found.

The first part of the article concerns the indices generalizing the
Maslov index for a path of symplectic transformations. For a
symplectic manifold $(V,\omega)$ satisfying certain hypotheses, these
are well--defined maps
$$\mu_k:\pi_{2k-1}(\Symp(V,\omega))\longrightarrow\Z/ p\Z,$$
and the map $\mu_1$ coincides with the Maslov index. These indices are
also defined for a contact manifold $(M,\xi)$. In this case, under
suitable hypothesis on the contact structure, the indices are
well--defined maps
$$\mu_k:\pi_{2k-1}(\Cont(M,\xi))\longrightarrow\Z/ p\Z.$$
The value of the integer $p=1, \dots, \infty$ depends on a
characteristic number of the manifold. In the second part of the
article we introduce a second family of indices related to the space
of compatible almost complex structures; these indices in contact
topology are known to be useful to detect non--trivial elements in the
contactomorphism group, and we refer to \cite{CP} for an
application. They are referred to as homogeneous indices and their
construction is based on the evaluation of a family of contact or
symplectic structures at a point. These homogeneous indices are also
defined in the symplectic case.

The third part of the article consists of some explicit examples and concrete applications of the above mentioned families of indices. For instance, we prove the following two results
\begin{theorem}\label{thm:sph}
The map $\pi_j(\iota):\pi_j(U(n+1))\longrightarrow\pi_j(\Cont(\S^{2n+1},\xi_0))$ is injective for $j\leq 2n+1$.
\end{theorem}
This theorem recovers and generalizes the non--triviality of the Reeb
flow associated to $(\S^{2n+1},\xi_0)$ studied in \cite{CP}. Its
symplectic counterpart is as follows.
\begin{theorem}\label{thm:proj}
The map $\pi_j(\iota):\pi_j(\P U(n+1))\longrightarrow\pi_j(\Symp(\CP^n,\omega_{FS}))$ is injective for $j\leq 2n+1$.
\end{theorem}
The fact that the map is an inclusion in rational homotopy has been proven in \cite{Re} with the use of symplectic characteristic classes. The proof we provide is simpler and emphasizes the topological (as opposed to symplectic) character of the result. There is also an application to small dual varieties, Proposition \ref{prop:lefschetz}, and to the space of contact structures isotopic to the unique tight contact manifold $(\S^1\times\S^2,\xi_{st})=\partial(\S^1\times D^3,\lambda_0)$ in Proposition \ref{prop:S1S2} and Corollary \ref{cor:S1S2}.

The article is organized as follows. Section \ref{sec:maslov}
introduces the higher Maslov indices in the contact and symplectic
case. An alternative description is given in Subsection
\ref{ssec:algebraic}. Section \ref{sec:homogeneous} details the
construction of the homogeneous indices. Section \ref{sec:app}
contains applications of both families of indices. In particular, we
establish Theorems \ref{thm:sph} and \ref{thm:proj} among other
results.

{\bf Acknowledgements}: We are grateful to Ignasi Mundet and Old\v{r}ich Sp\'{a}\v{c}il for valuable discussions.

\section{Maslov Indices}\label{sec:maslov}

In this Section we introduce the higher--rank Maslov indices. They can
be defined simultaneously for contact and symplectic transformations
since the main ingredient is a symplectic vector bundle with a certain
trivialization along a subset of the base manifold. For the sake of
simplicity, we consider first the symplectic symplectic case. The
definition for contact manifolds is then immediate.
\subsection{Indices on the group of symplectomorphisms}\label{ssec:sympmaslov}

Let $(V,\omega)$ be a symplectic manifold with $c_k(V,\omega)=0$ and let
$G=\Symp(V,\omega)$ be the connected component of the identity in the
group of symplectomorphisms of $V$. We will construct a map
$$\mu_k:\pi_{2k-1}(G)\longrightarrow\Z.$$
The vanishing of the $k$--th Chern class can be avoided; this will be
explained later. Let us spell out the construction of $\mu_k$.

We can fix a unitary frame $\lambda=\{e_1,\ldots,e_{n-k+1}\}$ of the
symplectic bundle $TV$ along the $2k$--skeleton $V^{(2k)}$ and a point
$p\in V^{(2k)}$. This frame exists because the unitary group $U(n)$ is
a maximal compact subgroup of the symplectic group $Sp(2n,\R)$ and we
have assumed $c_k(TV,\omega)=0$. (See Problem 14-C in \cite{MiS}.)

Consider a representative $s:\S^{2k-1}\longrightarrow G$ of a homotopy
class $[s]\in\pi_{2k-1}(G)$ and the evaluation map
$ev_p:G\longrightarrow V$ defined as $ev_p(\varphi)=\varphi(p)$. The
image $S$ of the map $ev_p\circ s:\S^{2k-1}\longrightarrow V$ is a
$(2k-1)$--sphere which is assumed to lie in $V^{(2k)}$. This can
always be achieved after a homotopy of $ev_p\circ s$. The bundle $TV$
is trivial along $S$, a trivialization can be obtained as
follows. Extend the frame $\lambda(p)$ at $p$ to
$\Lambda(p)=\lambda(p)\cup\{e_{n-k+2}(p),\ldots,e_n(p)\}$ and just
define $f_i(s(p))=s_*(e_i(p))$ the push--forward of the frame
$\Lambda(p)$ to $s(p)$ through $s$. This gives a global frame
$\kappa=\{f_1,\ldots,f_n\}$ of the bundle $TV|_S\longrightarrow
S$.
Next observe that the $(n-k)$-frame $\{f_1,\ldots,f_{n-k}\}$ extends,
up to homotopy, to $V^{(2k)}$. This is a consequence of the fact that
the Stiefel manifold $\mathbb{V}_{n-k}(\C^n)$ is $2k$--connected and
hence the obstruction classes to such an extension vanish. Hence,
without loss of generality, we can assume that this frame is defined
on $V^{(2k)}$.

\begin{lemma}\label{lem:htpyframe}
  There exists a homotopy $\{\tau_s\}_{s\in[0,1]}$
  of $(n-k+1)$--frames
  of $TV$
  along $V^{(2k)}$
  such that $\tau_0=\{e_1,\ldots,e_{n-k},e_{n-k+1}\}$
  and $\tau_1=\{f_1,\ldots,f_{n-k},v\}$
  for non--vanishing section $v\in\Gamma(TV)$.
\end{lemma}
\begin{proof}
  Two $(n-k)$--frames $\{e_1,\ldots,e_{n-k}\}$ and
  $\{f_1,\ldots,f_{n-k}\}$ in $TV|_{V^{(2k)}}$ can be connected
  through $(n-k)$--frames because, as has been mentioned above, the
  Stiefel manifold $\mathbb{V}_{n-k}(\C^n)$ is $2k$--connected and
  hence the obstruction classes lie in
  $H^i(V^{(2k)};\pi_i(\mathbb{V}_{n-k}(\C^n)))\cong\{0\}$ for all
  $0\leq i\leq 2k$.

Consider a homotopy $\tau_s:V^{(2k)}\longrightarrow \mathbb{V}_{n-k}(TV)$ with $\tau_0=\{e_1,\ldots,e_{n-k}\}$ and $\tau_1=\{f_1,\ldots,f_{n-k}\}$ and a hermitian metric on the bundle $TV$. Then at any $q\in V^{(2k)}$ we can pointwise parallel transport the section $e_{n-k+1}(q)$ along the orthogonal complements of the subspaces spanned by $\tau_s(q)$. This defines a vector $v_s(q)$ in the orthogonal complement of $\tau_s(q)$ and hence a non--vanishing section $v:V^{(2k)}\longrightarrow TV/\langle\tau_s\rangle$. We can thus extend the homotopy $\tau_s$ of $(n-k)$--frames to a homonymic homotopy of $(n-k+1)$--frames by adding the section $v_s$ to the $(n-k)$--frame. Note that the choice of the metric is unique up to homotopy and thus the extension is canonical in homotopy. This extension $\{\tau_s\}_{s\in[0,1]}$ satisfies the requirements of the statement.
\end{proof}

We apply Lemma \ref{lem:htpyframe} and obtain a non--vanishing section $v:S\longrightarrow TV\cong\C^n(f_1,\ldots,f_n)$ which lies on the complex rank $k$ trivial subbundle $W\longrightarrow S$ spanned by $\{f_{n-k+1},\ldots,f_n\}$. The section $v$ descends to a non--vanishing section $v:S\longrightarrow TV/\langle f_1,\ldots,f_{n-k}\rangle\cong W$, since $v$ is unitary it is a map to the sphere bundle $v:S\longrightarrow\S(W)\cong S\times \S^{2k-1}$. Denote by $\pi$ be the projection onto the second factor, we have obtained a map
\begin{equation}\label{eq:maslov}
\vartheta:\S^{2k-1}\stackrel{s}{\longrightarrow} \Symp(V,\omega)\stackrel{ev_p}{\longrightarrow} S \stackrel{v}{\longrightarrow} S\times\S^{2k-1}\stackrel{\pi}{\longrightarrow}\S^{2k-1}.
\end{equation}
The map $\vartheta$ has been constructed using a point $p\in V$, an $(n-k+1)$--frame $\lambda$, a map $s:\S^{2k-1}\longrightarrow\Symp(V,\omega)$ and performing two homotopies. Let us write $\vartheta(p,s,\lambda)$ to emphasize the choices. The degree of this map is an integer, this integer will be defined to be the Maslov index associated to the homotopy class $[s]\in\pi_{2k-1}(\Symp(V,\omega))$. Regarding the choices in our construction, we have the following
\begin{lemma}
The homotopy class of the map $\vartheta(p,s,\lambda)$ is independent of $p$, the representative $s$ and the homotopies performed in its construction.
\end{lemma}
Denote by $\vartheta([s],\lambda)$ the map constructed in (\ref{eq:maslov}) for a choice of $(n-k+1)$--frame $\lambda$ and any choice of $p\in V^{(2k)}$ and representative $s$ in its homotopy class.
\begin{definition}\label{def:sympMaslovI}
The $k$th Maslov index associated to an $(n-k+1)$--frame $\lambda$ is the map
$$\mu_k:\pi_{2k-1}(\Symp(V,\omega))\longrightarrow\Z,\quad \mu_k([s])=\deg(\vartheta([s],\lambda)).$$
\end{definition}
The map $\mu_k$ can be defined if there exists an $(n-k+1)$--frame
$\lambda$ of $TV$ along a $2k$--skeleton $V^{(2k)}$. This requires the
hypothesis $c_k(V,\omega)=0$. However, the construction of the map
$\vartheta([s],\lambda)$ carries over word-for-word to the setting
where $\lambda$ is just an $(n-k+1)$--frame along the sphere
$S=(ev_p\circ s)(\S^{2k-1})$.

Suppose that $\pi_{2k-1}(V)=\{0\}$ or just
$ev_p \circ s=0 \in \pi_{2k-1}(V)$. Then a capping disk $\D^{2n}$ for
$S=\partial\D^{2n}$ provides a frame
$\lambda=\{e_1,\ldots,e_{n-k+1}\}$ of $TV|_S$. The degree of the map
we obtain depends on the choice of capping disk, the dependence goes
as follows. The degrees obtained by two different choices $D_0$ and
$D_1$ of capping disks for $S$ differ by the evaluation of
$c_k(V,\omega)$ along the $2k$--sphere $D_0\cup_S D_1$. Denote by
$N_k$ the minimal $k$th Chern number, i.e. the positive generator
$N_k$ of the image of the morphism
$c_k(V,\omega):\pi_{2k}(V)\longrightarrow\Z$. Hence we can define the
map $\vartheta([s],\lambda)$ as in \ref{eq:maslov}. This time the map
is independent of $\lambda$ as long as we consider its degree to lie
in $\Z/N_k\Z$.
\begin{definition}\label{def:sympMaslovII}
The $k$th Maslov index of a class $[s]\in\ker(\pi_{2k-1}(ev_p))\subset\pi_{2k-1}(\Symp(V,\omega))$ is the map
$$\mu_k:\ker(\pi_{2k-1}(ev_p))\longrightarrow\Z/N_k\Z,\quad \mu_k([s])=\deg(\vartheta([s])).$$
\end{definition}

Definitions \ref{def:sympMaslovI} and \ref{def:sympMaslovII} complement each other. We refer to the maps in Definition \ref{def:sympMaslovI} as A--type Maslov indices and to the maps in Definition \ref{def:sympMaslovII} as B--type Maslov indices. In a manifold with both $c_k(V,\omega)=0$ and $\pi_{2k-1}(V)=\{0\}$ the element in $\Z/N_k\Z$ from Definition \ref{def:sympMaslovII} coincides with the reduction modulo $N_k$ of the integer from Definition \ref{def:sympMaslovI}. See Subsection \ref{ssec:algebraic} for more details.
\subsection{Indices on the group of contactomorphisms}\label{ssec:contmaslov} Let $(M,\xi)$ be a $(2n+1)$--dimensional contact manifold with $c_k(\xi)=0$ and $G=\Cont(M,\xi)$. We describe the construction of a map
$$\mu_k:\pi_{2k-1}(\Cont(M,\xi))\longrightarrow\Z$$
using the same technique explained in Subsection \ref{ssec:sympmaslov}. The procedure defining $\mu_k$ for the group of symplectomorphisms only requires a symplectic bundle with a unitary frame along the appropriate skeleton of the symplectic manifold. The contact distribution $\xi$ can be endowed with a symplectic 	structure: given a contact form $\alpha$ such that $\xi=\ker(\alpha)$, the $2$--form $d\alpha$ is a symplectic structure on the bundle $\xi\longrightarrow M$. The conformal symplectic structure of this bundle does not depend on the choice of $\alpha$ and thus the $k$th Chern class $c_k(\xi)$ of $\xi$ is well--defined.

Fix a frame $\lambda=\{e_1,\ldots,e_{n-k+1}\}$ for the rank--$n$ complex bundle $\xi$ along the $2k$--skeleton $M^{(2k)}$ and a point $p\in M^{(2k)}$. Consider a representative $s:\S^{2k-1}\longrightarrow G$ of a homotopy class $[s]\in\pi_{2k-1}(G)$. The construction in Subsection \ref{ssec:sympmaslov} yields a sphere $S=\im(ev_p\circ s)$ and a unitary section $v:S\longrightarrow\xi\cong\C^n(f_1,\ldots,f_n)$ whose image lies in the subbundle $\C^n(f_{n-k+1},\ldots,f_n)$. This gives a section of the sphere bundle of the rank--$k$ trivial quotient bundle $W\cong\xi/\langle f_1,\ldots,f_{n-k}\rangle$ and thus a map
\begin{equation}\label{eq:contmaslov}
\vartheta:\S^{2k-1}\stackrel{s}{\longrightarrow} \Cont(V,\omega)\stackrel{ev_p}{\longrightarrow} S \stackrel{v}{\longrightarrow} S\times\S^{2k-1}\stackrel{\pi}{\longrightarrow}\S^{2k-1}.
\end{equation}
The degree of $\vartheta$ is the required integer. The construction also works if $c_k(\xi)$ does not vanish but there exists a unitary frame $\lambda$ along the sphere $S$. The discussion preceding Definition \ref{def:sympMaslovII} applies and we can define the maps in both cases. Denote by $N_k$ the $k$th minimal Chern number of $\xi$.
\begin{definition}\label{def:contMaslov}
The $k$th Maslov index associated to an $(n-k+1)$--frame $\lambda$ is the map
$$\mu_k:\pi_{2k-1}(\Cont(V,\omega))\longrightarrow\Z,\quad \mu_k([s])=\deg(\vartheta([s],\lambda)).$$
When the class $[s]$ belongs to $\ker(\pi_{2k-1}(ev_p))\subset\pi_{2k-1}(\Cont(M,\xi))$ the map
$$\mu_k:\ker(\pi_{2k-1}(ev_p))\longrightarrow\Z/N_k\Z,\quad \mu_k([s])=\deg(\vartheta([s])).$$
is also referred to as the $k$th Maslov index.
\end{definition}
The maps coincide in their common domain of definition up to reduction mod $N_k$. This is again a consequence of the geometric fact that the frame $\lambda$ allows us to define a well--defined map to $\Z$ whereas the dependence of such choice forces a reduction to values in $\Z / N_k \Z$ .
\subsection{Algebraic interpretation}\label{ssec:algebraic}
In this subsection we provide some intuition for the indices $\mu_k$ from the homotopy viewpoint and a cohomological interpretation. We begin with the definition of $\mu_k$ in terms of the homotopy groups of $U(n)$.

The construction of the Maslov index $\mu_1$ is based on the (square of the) determinant map $\det:U(n)\longrightarrow\S^1$. There are maps from $U(n)$ to other spheres: the transitive action of $U(n)$ on $\C^n$ induces a map $p^n:U(n)\longrightarrow\S^{2n-1}$. In linear algebra terms, $p^n(A)=\{\mbox{first column of }A\}$. The homotopy class of $p^n$ does not change if another column is chosen. Heuristically, we want to construct maps $U(n)\longrightarrow\S^{2k-1}$ that allow us to define $\mu_k$ the same way the determinant map defines $\mu_1$. The formal algebraic phenomenon behind this is that in rational homotopy (or rational cohomology) $U(n)$ is a product of odd dimensional spheres. This should provide a map $\pi_{2k-1}U(n)\longrightarrow\pi_{2k-1}(\S^{2k-1})
$. We can readily construct this map as follows.

The group $U(n)$ acts transitively on the unit sphere $\S^{2n-1}$ with stabilizer group isomorphic to $U(n-1)$. This yields the Serre fibration
\begin{equation}\label{eq:unitary}
U(n-1)\stackrel{i^n}{\longrightarrow}U(n)\stackrel{p^n}{\longrightarrow}\S^{2n-1}.
\end{equation}
This implies that the maps $\pi_*(i^n):\pi_*(U(n-1))\longrightarrow\pi_*(U(n))$ are isomorphisms for $*\leq2n-3$. 

\begin{proposition}\label{prop:htpyU(n)} Let $k,n\in\N$ be such that $k\leq n-1$.
Then there exists a sequence of isomorphisms
$$\pi_{2k-1}U(n)\longrightarrow\pi_{2k-1}U(n-1)\longrightarrow\ldots\longrightarrow\pi_{2k-1}U(k)\longrightarrow \frac1{(k-1)!}\pi_{2k-1}(\S^{2k-1})\cong\Z.$$\hfill$\Box$
\end{proposition}
There are several reasons for the $(k-1)!$ factor on the rightmost
isomorphism. For instance, among these are the description of the
Chern character in the reduced K--theory of a sphere and the
Atiyah--Singer theorem for the twisted signature operator. See
\cite{BH,Pe} for a more classical approach.

Suppose that $V=\C^n$ and we consider the group $G\simeq U(n)$ of
linear symplectomorphisms. Then the index $\mu_k$ is precisely the
image via the chain of isomorphisms described in Proposition
\ref{prop:htpyU(n)}. The definition in the non--linear context is
obtained by pointwise linearizing the family of symplectomorphisms and
considering the image via a parametric version of the chain of
isomorphisms. The frames are required to be fixed in order to identify the unitary maps between tangent spaces with the group $U(n)$. The description in Subsection \ref{ssec:sympmaslov} is more elementary and this homotopic interpretation is better understood once the construction in Subsection \ref{ssec:sympmaslov} has been introduced. This homotopic interpretation also holds for the contact case where fiberwise the role of $\C^n$ is taken by $\xi_p$ but not~$T_pV$.

Now that the homotopical nature of $\mu_k$ has been clarified, we
establish a formula that allows one to compute $\mu_k$ in certain
situations. Let us consider a symplectic manifold $(V,\omega)$ and a
point $p\in V^{(2k)}$. The evaluation map
$ev_p:\Symp(V,\omega)\longrightarrow V$ can be used to detect
non--trivial elements in the homotopy groups of
$\Symp(V,\omega)$. Suppose that we have a class
$[s]\in\ker(\pi_{2k-1}(ev_p))\subset\pi_{2k-1}(\Symp(V,\omega))$, the
B--type Maslov index is then defined. This index $\mu_k([s])$ can
be computed with the appropriate cohomological data.

We introduce some notation. Given a map $[s]\in\pi_{2k-1}(\Symp(V,\omega))$ denote by $X_s$ the (isomorphism class of the) symplectic bundle
$$V\longrightarrow X_s\stackrel{\pi}{\longrightarrow}\S^{2k}$$
with clutching function $s$ and denote by $\V X_s=\ker d\pi$ the
vertical bundle of $\pi$. The Serre fibration $\pi$ does not
necessarily have a section $\s:\S^{2k}\longrightarrow X_s$. We can
find such a section $\s_\SU$ over the base $\S^{2k}$ with the
neighborhood $\SU$ of a point removed. The extension of $\s_\SU$ to
$\S^{2k}$ depends on the clutching map $s$. We can consider a point
$p\in\pi^{-1}(\partial\SU)$ in the pre--image of the boundary equator
$\partial(\S^{2k}\setminus\SU)$. Then the image of $ev_p(s)$ describes
a section over $p\in\pi^{-1}(\partial\SU)$. In the case where this
image is a contractible subspace we can extend $\s_\SU$ to $\s$ using
a homotopy of $ev_p(s)$ to a point. This is the geometric description
of the obstruction class in
$H^{2k}(\S^{2k},\pi_{2k-1}(V))\cong\pi_{2k-1}(V)$ for the
corresponding lifting problem.

The choice of such an extension, that is the choice of homotopy to a
point, is relevant. The $k$th Chern class of $V$ evaluated on a sphere
representing the difference between two capping disks results in an
ambiguity modulo $N_k$. This corresponds to the uniqueness obstruction
class in $H^{2k}(\S^{2k},\pi_{2k}(V))\cong\pi_{2k}(V)$.

The B--type index $\mu_k$ can then be computed with data intrinsic to $X_s$.

\begin{theorem}\label{thm:cohom} Consider a point $p\in V^{(2k)}$ and a class $[s]\in\ker(\pi_{2k-1}(ev_p))$. Then there exists a section $\s\in\Gamma(X_s)$ and the Maslov index satisfies
$$\mu_k\left([s]\right)=\int_{\S^{2k}}c_k(\s^*(\V X_s))\pmod{N_k}.$$
\end{theorem}
\begin{corollary}\label{cor:cohom}
Suppose that $c_k(V,\omega)=0$. Then the A-type index $\mu_k$ vanishes identically on the subgroup $\ker(\pi_{2k-1}(ev_p))\subset\pi_{2k-1}(\Symp(V,\omega))$.
\end{corollary}
\begin{proof}
Consider the inclusion $i:V\longrightarrow X_s$ of the fibre $V$. The Leray--Serre spectral sequence (or the Morse--Bott height function) and the section $s$ yields the isomorphism $$H_{2k}(X_s,\Z)\cong H_0(\S^{2k},\Z)\otimes H_{2k}(V,\Z)\oplus H_{2k}(\S^{2k},\Z)\otimes H_0(V,\Z).$$
The cohomology class $c_k(\V X_s)\in H^{2k}(X,\Z)$ pairs trivially
with $H_{2k}(\S^{2k},\Z)\otimes H_0(V,\Z)$ and pulls--back to
$i^*c_k(\V X_s)=c_k(TV)\in H^{2k}(V,\Z)$. The statement follows from
Theorem \ref{thm:cohom} with $N_k=\infty$.
\end{proof}

\subsection{Vanishing of the A-type indices in the symplectic case}
\label{ssec:symp}
An immediate consequence of Corollary \ref{cor:cohom} is that
$\mu_k=0$ for the A-type indices on the standard symplectic torus
$\T^{2n}$ when $k\geq 2$. In fact, the authors are not aware of any
example of a closed symplectic manifold $V$ with $c_k(TV)=0$ such that
an A-type index $\mu_k$, $k\geq 2$, does not vanish.

In contrast, the classical Maslov class $\mu_1$ on $\T^{2n}$ need not
be zero for a rather superficial reason. Namely, this index is
obviously non-zero for a global frame on $\T^{2n}$ ``twisted in a
coordinate direction''. However, as we prove below, one can make
$\mu_1$ vanish by a suitable choice of a frame. (For $\T^{2n}$, this
is the standard coordinate frame.) We do not know if this is always
the case when $V$ is closed and $c_1(TV)=0$, but, as we show in this
section, this ``triviality'' of $\mu_1$ is quite a common occurrence.

Let $(V^{2n},\omega)$ be a closed symplectic manifold.  We focus on
the A-type indices and whenever the index $\mu_k$ is considered we
assume that $c_k(TV)=0$. Denote by $\Ham=\Ham(V,\omega)$ the group of
the Hamiltonian diffeomorphisms of $V$. As is well known, $\mu_1=0$ on
$\pi_1(\Ham)$. This is an easy consequence of Arnold's
conjecture. Indeed, arguing by contradiction, assume that
$\mu_1(\varphi)\neq 0$ for some loop $\varphi$ in $\Ham$. Composing a
high iteration of $\varphi$ with the flow of a $C^2$-small autonomous
Hamiltonian, we obtain a path in $\Ham$ having no one-periodic orbits
with Conley--Zehnder index in the range $[-n,n]$, which is impossible
by Arnold's conjecture; see, e.g., \cite{Sa:LN} and references
therein.

Furthermore, recall that the flux map $\Phi$ associates to an element
$\varphi$ in the universal covering of $\Symp=\Symp(V,\omega)$ a
cohomology class $\Phi(\varphi)\in H^1(V;\R)$.  The value of this
class on the homology class of a loop $\gamma$ is equal to the
symplectic area of the cylinder swept by $\gamma$ under $\varphi$.
This map gives rise to an isomorphism between the quotients
$\Symp/\Ham$ and $H^1(V;\R)/\Gamma$ where the group
$\Gamma=\Phi(\pi_1(\Ham))$ is discrete; see \cite{Ba,LMP,On}. We
conclude that $\mu_1$ descends from $\pi_1(\Symp)$ to
$\pi_1(\Symp/\Ham)=\Gamma$ and vanishes on $\pi_1(\Symp)$ if and only
if it vanishes on $\Gamma$. Moreover, $\mu_k=0$ on $\pi_k(\Symp)$ for
$k\geq 2$ if and only if $\mu_k=0$ on $\pi_k(\Ham)$. (As has been
mentioned above, we do not have any example where this index with
$k\geq 2$ is non-trivial for any choice of the frame. It would be
extremely interesting to find such an example or show that in general
$\mu_k=0$ for $k\geq 2$.)

We are now in a position to state the vanishing results for
$\mu_1$. The key is the following, fairly standard, observation.

\begin{proposition}
Assume that $c_1(TV)=0$. Then
\begin{itemize}
\item[(1)] if $\omega|_{\pi_2(V)}=0$, every loop $\varphi$ in $\Symp$ with
  contractible orbits in $V$ has zero flux;

\item[(2)] if $V$ is Lefschetz, i.e., the multiplication by
  $[\omega]^{n-1}$ is an isomorphism $H^1(V;\R) \to H^{2n-1}(V;\R)$,
  every loop $\varphi$ in $\Symp$ with orbits homologous to zero in $V$ over
  $\R$ has zero flux.
\end{itemize}
\end{proposition}

\begin{proof}
  In (1), the flux $\Phi(\varphi)(\gamma)$ is the integral of
  $\omega$ over the torus swept by $\gamma$ under the flow of the
  loop. The loop swept by $\gamma(0)$ in $V$ is contractible. Let us
  attach (twice, with opposite orientations) a disk bounded by this
  loop to the torus. The integral of $\omega$ over the resulting
  sphere is equal to the integral over the torus (the flux) and is
  zero since $\omega$ vanishes on $\pi_2(V)$.

  To deal with (2), recall that the homology class of an orbit (over
  $\R$) is Poincar\'e dual to
  $\Phi(\varphi)\wedge[\omega]^{n-1}/\mathit{vol}(V)$. Thus, under the
  Lefschetz condition, the orbit is homologous to zero if and only if
  it has zero flux.
\end{proof}

\begin{corollary}
  In the setting of (1) or (2), $\varphi$ is homotopic to a loop in
  $\Ham$, and hence $\mu_1(\varphi)=0$.
\end{corollary}

\begin{proof}
  We already know that $\mu_1=0$ on $\pi_1(\Ham)$,
  and hence $\mu_1(\varphi)=0$ when $\varphi$ is homotopic to a loop
  in $\Ham$.
\end{proof}

\begin{corollary}[Triviality]
  Assume that the evaluation map is onto $\pi_1(V)$ when
  $\omega|_{\pi_2(V)}=0$ or that it is onto $H_1(V;\Z)/\mathit{Tors}$
  if $V$ is Lefschetz. Then $\mu_1$ can be made identically zero by a
  suitable choice of trivialization.
\end{corollary}

\begin{proof}
  By changing the frame, one can adjust $\mu_1$ by an arbitrarily
  homeomorphism $\pi_1(V)\to \Z$ or $H_1(M;\Z)/\mathit{Tors}\to \Z$. The result
  follows. 
\end{proof}

\section{Homogeneous Indices}\label{sec:homogeneous}
In this Section we introduce a second family of indices. These are
maps from the space of contact or symplectic structures instead of
their groups of transformations. The target spaces of these maps will
be the homotopy groups of the space of almost complex structures. The
basic facts about this space are stated in Subsection
\ref{ssec:acs}. Then the maps are defined in Subsection
\ref{ssec:homind}. These homogeneous indices are quite simple and make
a natural partner for the Maslov indices defined in Section
\ref{sec:maslov} and are worth to be treated in a systematic fashion. 

Consider a smooth manifold $X$, it will denote either a symplectic manifold $(V,\omega)$ or a contact manifold $(M,\xi)$. Let $G(X)$ denote either $\Symp(V,\omega)$ or $\Cont(M,\xi)$ intersected with $\Diff_0(X)$\footnote{Remark that $G(X)$ may not be connected.}. Define the spaces
$$\SS(V,\omega)=\{\eta\in\Omega^2(V): [\omega]=[\eta]\mbox{ and }\eta\mbox{ is symplectic}\},$$
$$\SC(M,\xi)=\{\SD:\SD\mbox{ is a contact structure isotopic to }\xi\}.$$
Denote the connected component containing either $\omega$ or $\xi$ of these spaces by $\SK_0(X)$. The parametric version of theorems of J. Moser and J.W. Gray \cite{Gr,Mo} implies the following result.
\begin{lemma}\label{lem:serre}
The group $\Diff_0(X)$ acts transitively on $\SK_0(X)$ with stabilizer isomorphic to $G(X)$. There exists a Serre fibration
$$G(X)\longrightarrow \Diff_0(X)\stackrel{\pi}{\longrightarrow}\SK_0(X)\mbox{ where }\pi(\varphi)=\varphi^*\omega\mbox{ or }\varphi_*\xi.$$
\end{lemma}
Given the knowledge of the homotopy type of $\Diff_0(X)$ the study of
the homotopy groups of the spaces $G(X)$ or those of $\SK_0(X)$ is
equivalent. In particular, one can extract information about $G(X)$
through the study of $\SK_0(X)$. For instance, this has been exploited
in \cite{AM,Bu,Kr,Mc} for the symplectic case and studied in
\cite{CP,GK,GG1} in the contact case. The space $\SK_0(X)$ can be
analyzed through the evaluation map at a point with the use of the
space of linear almost complex structures. The following Subsection
gives a description of the space of linear complex structures.
\subsection{Almost complex structures}\label{ssec:acs} Let $E$ be a $2n$--dimensional real vector space. Consider the space of linear complex structures
$$\SA_l(E)=\{j\in\End(V):j^2=-1\}.$$
Let us fix a linear isomorphism of the vector space $E$ with
$\R^{2n}$. The special orthogonal group $SO(2n)$ acts transitively on $\SA_l(\R^{2n})$ with stabilizer group isomorphic to $U(n)$. This gives a diffeomorphism $\SA_l(E)\cong SO(2n)/U(n)$.
These homogeneous spaces have a well--understood geometry. In
low--dimensions their smooth description is quite simple:
$SO(4)/U(2)\cong\S^2$, $SO(6)/U(3)\cong\CP^3$ and $SO(8)/U(4)\cong
\Gr^+_2(\R^8)$. In general, their stable homotopy groups can be
computed by using the results of R. Bott, \cite{Bo}.

\begin{lemma}\label{lem:homot} Let $\SA_{l}(\R^{2n})$ be the space of linear complex structures.\\
In the stable range $1\leq k\leq 2n-2$, the homotopy groups are
$$\pi_k(\SA_{l}(\R^{2n})) = \begin{cases} \Z_2 &\mbox{if } k \equiv 0,7 \pmod{8},\\
\Z & \mbox{if } k \equiv 2 \pmod{4},\\
0 & \mbox{otherwise}.
\end{cases}$$
The inclusion $\SA_{l}(\R^{2n})\longrightarrow \SA_{l}(\R^{2n+2})$ is an isomorphism of the stable homotopy groups.
\end{lemma}

\begin{proof}
  This is a consequence of the Bott Periodicity Theorem which asserts
  the homotopy equivalence $\Omega O \simeq O/U$ and provides the
  stable homotopy groups of $O$. The shift in the degree with respect
  to $\pi_*(O)$ is caused by the general isomorphism
  $\pi_*(\Omega X)\cong\pi_{*+1}(X)$ provided by the path
  fibration. This gives the groups in the statement.

The homotopy cofiber of the map $\SA_{l}(\R^{2n})\longrightarrow \SA_{l}(\R^{2n+2})$ is diffeomorphic to $\S^{2n}$, hence it an isomorphism in the stable homotopy groups.
\end{proof}

\noindent The homogeneous indices take values in the homotopy groups
of the homogeneous spaces $\SA_l(E)$. The previous Lemma gives a
simple description for the cases in the stable range. For instance,
the index landing in the second homotopy group
$\pi_2(\SA_l(E))\cong\Z$ is integer valued.

\noindent The vector space $E$ that is used to define the homogeneous indices will be the tangent space of a manifold at a point. The tangent space of a symplectic manifold is even dimensional and the previous description applies. A contact manifold is odd--dimensional and there exists no linear complex structure on an odd--dimensional real vector space. This can be fixed as follows.

Let $F$ be a ($2n+1$)--dimensional vector space with an inner product $(\cdot,\cdot)$. One can have a linear complex structure either by considering a hyperplane $H$ and a linear complex structure on it or artificially add a one--dimensional vector space $L$ and endow $F\oplus L$ with a linear complex structure. Both options lead to the notion of a linear almost contact structure. The former description leads to the following definition for the space of linear almost contact structure:
$$\SA_l(F):=\{H\in\Gr_{2n}^+(F),j\in\End(H):j^2=-1\}.$$
This naturally gives a complex structure on $F\oplus L$ by declaring $H^\perp\oplus L\cong\C$. The space $\SA_l(F)$ can be identified with the homogeneous space $SO(2n+1)/U(n)$. It has a nice description in terms of linear almost complex structures on $F\oplus L$:

\begin{lemma}\label{lem:iso}
There exists a diffeomorphism $\SA_l(\R^{2n+1})\cong\SA_l(\R^{2n+2})$.
\end{lemma}

\noindent The proof is an exercise in the theory of homogeneous spaces \cite{Ge}. In particular Lemmas \ref{lem:homot} and \ref{lem:iso} allow to compute the stable homotopy groups of $\SA_l(F)$.
\subsection{Definition of the indices}\label{ssec:homind} Let us define the homogeneous indices. Denote by $\SA(X)$ either of the two spaces
$$
\SA(V)=\{\J\in\End(TV),\J^2=-\mbox{id}\}
\textrm{ or }\SA(M)=\{(\xi,\J):\xi\in Dist(M),\J\in\End(\xi),\J^2=-\mbox{id},\J\mbox{ positive}\}.$$
and by $\SA\SC(X)$ either of the two spaces
$$\SA\SC(V,\omega)=\{(\eta,\J):\eta\in\SS(V,\omega),\J\mbox{ is compatible with }\omega\},$$
$$\SA\SC(M, \xi)=\{(\eta,\J):\eta \in \SC(M, \xi),\J\in\End(\eta),\J^2=-\mbox{id}, \J \mbox{ compatible with }\alpha\mbox{ such that }\eta=\ker\alpha\}.$$

The forgetful projection $\SA\SC(X)\longrightarrow\SK(X)$ has contractible homotopy fibres and thus there exists a homotopy inverse $\i:\pi_*\SK(X)\longrightarrow\pi_*\SA\SC(X)$. Consider the composition $\tilde\i:\pi_*\SK(X)\longrightarrow\pi_*\SA(X)$ of this homotopy inverse $\i$ with the map induced by the projection $\rho:\SA\SC(X)\longrightarrow\SA(X)$.

\noindent Fix a point $p\in X$, an oriented frame $\tau:T_pX\stackrel{\simeq}{\longrightarrow}\R^{N}$ and define the evaluation map
$$ev_{(p,\tau)}:\SA(X)\longrightarrow\SA_l(\R^N),\quad ev_{(p,\tau)}(\J)=\tau_*\J.$$

\begin{definition}\label{def:homind}
The homogeneous index associated to $(p,\tau)$ is the map
$$\varepsilon_*=\pi_*ev_{(p,\tau)}\circ\tilde{\i}:\pi_*(\SK(X))\longrightarrow\pi_*(\SA_l(\R^N))$$
\end{definition}

The homogeneous indices do not depend on the choice of data $(p,\tau)$.
\begin{lemma}\label{lem:sympinvar}
The induced evaluation maps in homotopy
$$\varepsilon_*= \pi_*e_{(p,\tau)}\circ\tilde{\i}:\pi_*\SK(X)\longrightarrow\pi_*\SA_l(\R^N)$$
are independent of the choice of $(p,\tau)$.
\end{lemma}
\begin{proof}
Note that the bundle of frames is connected. Then a path between two choices $(p,\tau_p)$ and $(q,\tau_q)$ induces a homotopy between the maps $ev_{(p,\tau_p)}$ and $ev_{(q,\tau_q)}$, hence $\varepsilon_*$ is independent of the choice of $(p,\tau)$.
\end{proof}
\section{Applications}\label{sec:app}
\subsection{Small dual varieties}\label{ssec:small} In this subsection
we exhibit examples of symplectic manifolds with loops of
symplectomorphism whose B-type Maslov index $\mu_1$ is non--zero. Theorem \ref{thm:cohom} can be used for this purpose. Elements in the homotopy groups of the group of symplectomorphisms can be produced using complex vector bundles or more generally holomorphic fibrations. The underlying symplectic topology of the fibres of a holomorphic fibration remains invariant whereas their complex tends to vary. See \cite{Kr,Se} for examples. We will also use a construction from algebraic geometry.

A polarized projective manifold $X\subset\P^N$ is said to have small
dual if its dual variety $X^*$ has complex codimension at least 2 in
$(\P^N)^*$. The symplectic topology of small dual projective manifold
is studied in \cite{BJ}. It is a classical result that projective
manifold with a small dual of $\codim(X^*)=d+1$ is $\P^d$--ruled. In
particular, it is simply connected; see \cite{Ei}. It is also easy to
prove that there exists a Lefschetz pencil $(X,B,\pi)$ with simply
connected fibers and no critical points. The existence of a Lefschetz
pencil allows us to construct a symplectic fibration over $\P^1$ and
obtain a loop of symplectomorphisms as the monodromy of such
fibration.

\begin{proposition}\label{prop:lefschetz}
Let $(X^{2n},\omega)$ be a polarized projective manifold with small
dual, $n\geq3$. Consider a non--critical Lefschetz pencil
$(X,B,\pi)$. Then the induced blow--up fibration $(\tilde X_B,\pi)$
with fiber the hyperplane section $W=X\cap H$ has as monodromy loop
$\gamma$ a loop of symplectomorphisms with $\mu_1(\gamma)=-1 \pmod {N_1}$.
\end{proposition}
\begin{proof}
Fix a non--critical Lefschetz pencil $(X,B,\pi)$ that necessarily has $B\neq\emptyset$. The fibres of $\pi$ are given by hyperplane sections of $X$. Consider the blow--up $\widetilde{X}_B$ of $X$ along $B$. It has two projections
$$X\stackrel{\sigma}{\longleftarrow}\widetilde{X}_B\stackrel{\pi}{\longrightarrow}\P^1.$$
Observe that $\widetilde{B}=\sigma^{-1}(B) \stackrel{\pi_B}{\to} B$ is a $\P^1$--bundle over $B$. Indeed the fiber over $b\in B$ is $\P(\nu_{B,b})\cong\P(\C^2)$, where $\nu_B$ is the normal bundle of $B$ in $X$. Hence the choice of a point $b\in B$ in the base locus provides a section $s_b$ of $\pi$. We shall compute the Maslov index of the monodromy loop $\gamma$ using such a section and Theorem \ref{thm:cohom}.

\noindent Consider a point $q\in\widetilde{B}$. The tangent space of $W=\pi^{-1}(\pi(q))$ at that point splits as
$$T_q\widetilde{X}_B\cong T_q\widetilde{B}\oplus \nu_{\widetilde{B},q}.$$
This establishes a bundle isomorphism with basis $\widetilde{B}$. We can further decompose, over $\widetilde{B}$, as
$$(T\widetilde{X}_B)_{|\widetilde{B}}\cong (\ker (\pi_B)_*)\oplus \pi_B^*(TB) \oplus \nu_{\widetilde{B}}.$$
It is clear that $V(\widetilde{X}_B)_{|\widetilde{B}} \simeq \pi_B^*(TB) \oplus \nu_{\widetilde{B}}$ and $(\nu_{W})_{|\widetilde{B}}=  \ker (\pi_B)_*$. In particular, it restricts to a bundle in the image of the section $s_{\sigma(q)}:\P^1\longrightarrow \widetilde{B} \subset \widetilde{X}_B$. Since $\sigma(s_{\sigma(q)})=q$, the bundle $\pi_B^*(TB)$ along $\im(s_{\sigma(q)})$ pulls--back via $s_{\sigma(q)}$ to the trivial bundle over $\P^1$. The normal bundle $\nu_{\widetilde{B}}$ restricts to $\O(-1)$ over the section. Hence
$$\mu_1(\gamma)\equiv\langle c_1(V(\widetilde{X}_B)),[\P^1]\rangle\equiv \langle c_1(\O(-1)),[\P^1]\rangle\equiv-1.$$
\end{proof}

Since $X$ is simply connected, $W=X\cap H$ is also simply connected
for $n \geq 3$ by the Lefschetz hyperplane theorem


\subsection{Standard contact spheres} In this Subsection we prove Theorem \ref{thm:sph} stating that the inclusion $j:U(n)\longrightarrow\Cont(\S^{2n-1},\xi_0)$ induces an injection in the homotopy groups $\pi_i$ for $i=1,\ldots,2n-1$.


Consider standard Euclidean complex space $(\C^n,i)$ and its real unit sphere $\S^{2n-1}=\{z\in\C^n:|z|^2=1\}$. The contact structure $\xi_0$ is the complex $(n-1)$--distribution $\xi_0=T\S^{2n-1}\cap i(T\S^{2n-1})$. A unitary linear operator $A:\C^n\longrightarrow\C^n$ commutes with the complex structure $i$ and thus preserves the distribution $\xi_0$, i.e. $A_*(\xi_0(z))=\xi_0(Az)$. Hence $U(n)$ includes into $\Cont(\S^{2n-1},\xi_0)$.

Let us compute the Maslov indices $\mu_k$ on the subgroups 
$$
\pi_*(j)(\pi_{2k-1}(U(n)))\subset\pi_{2k-1}(\Cont(\S^{2n-1}),\xi_0).
$$
This will prove the non--triviality of the image of the generators of $\pi_{2k-1}(U(n))$ in the contactomorphism group for $k=1,\ldots,n$. In this range the homotopy of $U(n)$ is simple:
\begin{lemma}$($\cite{Bo}$)$\label{lem:htpyunitary}
Let $k,n\in\N$ with $k\leq 2n-1$. Then the following isomorphisms hold
$$\pi_kU(n)\cong\Z\mbox{ if $k$ is odd, }\quad\pi_kU(n)\cong\{0\}\mbox{ if $k$ is even}.$$
\end{lemma}
This dichotomy is the geometric reason for which the indices $\mu_k$ are only defined in the odd homotopy groups (or algebraically, the reason for which the Chern classes only belong to the even dimensional cohomology). Theorem \ref{thm:sph} will follow from the right understanding of Lemma \ref{lem:htpyunitary}. Let us provide an appropriate geometric description of the homotopy groups $\pi_{2k-1}(U(n))$.

The manifold $\S^{2n-1}\setminus\{q\}\cong\D^{2n-1}$ is contractible and then (\ref{eq:unitary}) implies that
\begin{equation}\label{eq:trivialfib}
p^n:(p^n)^{-1}(\S^{2n-1}\setminus\{q\})\longrightarrow \S^{2n-1}\setminus\{q\}
\end{equation}
is a trivial $U(n-1)$--fibration over a disk $\D^{2n-1}$. It can be trivialized with a unitary connection and a contraction of $\D^{2n-1}$ to a point. For instance, the height function on $\S^{2n-1}$ is a Morse function whose gradient flow can be lifted to $U(n)\longrightarrow\S^{2n-1}$ (or we can also consider the height function as a Morse--Bott function on $U(n)$) thus providing an isomorphism of (\ref{eq:trivialfib}) with $U(n-1)\times\D^{2n-1}$. This is a geometric interpretation of Proposition \ref{prop:htpyU(n)} and it is implicitly used in the proof of Theorem \ref{thm:sph}.

{\it Proof of Theorem \ref{thm:sph}}: Consider the coordinates
$(z_1,\ldots,z_n,\overline{z}_1,\ldots,\overline{z}_n)\in\C^n\setminus\{0\}$
and the unitary frame $\lambda=\{e_1,\ldots,e_n\}$ associated to the
vectors
$\{z_1\partial_{z_1}+\overline{z}_1\partial_{\overline{z}_1},\ldots,z_n\partial_{z_n}+\overline{z}_n\partial_{\overline{z}_n}\}$. The
contact structure $\xi_0\longrightarrow\S^{2n-1}$ can be described as
the kernel of the $1$--form 
$$
\alpha=\frac{i}{2}\cdot\sum_{i=1}^nz_id\overline{z}_i-\overline{z}_idz_i.
$$ 
Let us compute the image $\mu_k$ of an element $[s]\in\pi_{2k-1}(U(n))$ using this frame $\lambda$.

Fix a point $p\in\S^{2n-1}$ and let $\sigma\in\S^{2k-1}$. In order to determine the frame $\{f_1,\ldots,f_n\}$ associated to the sphere of contactomorphisms $s$ we have to push--forward the frame $\lambda(p)$ for each $s(\sigma)$. In this case any of the maps $s(\sigma)$ can be considered as a linear map on $\C^n$ and thus represents its own differential. The frames $s(\sigma)_*\lambda$ can be represented by unitary matrices with respect to the frame $\lambda$ and the homotopy that carries $\{e_1,\ldots,e_{n-k}\}$ to $\{f_1,\ldots,f_{n-k}\}$ in the construction for Definition \ref{def:sympMaslovI} can be considered as a path of matrices. Instead of deforming the frame $\{e_1,\ldots,e_{n-k},e_{n-k+1}\}$ to $\{f_1,\ldots,f_{n-k},v\}$ as in Lemma \ref{lem:htpyframe} we will deform $s(\sigma)_*\lambda=\{f_1,\ldots,f_{n-k}\}$ to $\{e_1,\ldots,e_{n-k}\}$. In the linear algebra viewpoint, this is deforming the matrices $s(\sigma)_*\lambda$ to matrices of the form
$$\left(
\begin{array}{ccccc}
I_{n-k} & A_1 \\
0 & A_2\\
\end{array}
\right).$$
Then the homotopy class of the section $v$ in Lemma \ref{lem:htpyframe} is represented by any of the columns of the matrix $(A_1|A_2)^t$. Hence $\mu_k$ is precisely the composition of the chain of isomorphisms in Proposition \ref{prop:htpyU(n)} and we conclude the statement of Theorem \ref{thm:sph}.\hfill$\Box$

In particular, Theorem \ref{thm:sph} follows. Theorem \ref{thm:proj} is also deduced with this same method, and we leave the details to the reader.

\begin{remark}
  The proof of Theorem \ref{thm:sph} hinges on the fact $\mu_k\neq 0$
  for the A-type indices on $\S^{2n-1}$ in the stable range $k\leq n$;
  for $\mu_k([s])=1$ where $[s]$ is the image of the generator of
  $\pi_{2k-1}(U(n))$. In a similar vein, the essence of the proof of
  Theorem \ref{thm:proj} is the fact that $\mu_k\neq 0\pmod {N_k}$ for
  the $A$-type indices on $\CP^{n}$ as long as $k$ is again in the
  stable range. This is also proved by evaluating $\mu_k$ on the
  generator of $\pi_{2k-1}(\P U(n))$.
\end{remark}

\subsection{Two classes in $\SC(\S^1\times\S^2,\xi_{st})$}\label{ssec:S1S2} In this subsection we compute the contact homogeneous indices for a pair of elements in the second homotopy group of the space of contact structures isotopic to the standard contact structure in $\S^1\times\S^2$. Consider the smooth manifold $\S^1\times\S^2$ with coordinates $(\theta;x,y,z)\in\S^1\times\R^3$ and its unique tight contact structure $\xi$. See \cite{Ge} for details. A possible contact form is
$$\alpha=zd\theta+xdy-ydx.$$
We introduce two classes $[S],[\delta]\in\pi_2\SC(\S^1\times\S^2,\xi)$ and use the contact homogeneous index to prove they are distinct. We begin with $[\delta]$.

\noindent Consider polar coordinates $(\phi,\varphi)\in[0,\pi]\times[0,2\pi]$ in $\S^2$. A point $(\phi,\varphi)\in\S^2$ determines an axis $l(\phi,\varphi)$ in $\R^3$. Denote by $R^\sigma_{\phi,\varphi}\in SO(3)$ the rotation of angle $\sigma$ along the axis $l(\phi,\varphi)$. The sphere $\delta$ is not directly defined in $\SC(\S^1\times\S^2,\xi)$. We first define the map
$$\delta:\S^2\longrightarrow\Diff(\S^1\times\S^2),\quad (\phi,\varphi)\longmapsto\delta(\phi,\varphi)(\theta;x,y,z)=(\theta; R^\theta_{\phi,\varphi}(x,y,z)).$$
It is known that $\Diff(\S^1\times\S^2)\simeq SO(2)\times SO(3)\times\Omega SO(3)$. This is proven in \cite{Ha}. In particular
$$[\delta]\in\pi_2\Diff(\S^1\times\S^2)\cong\pi_2(\Omega SO(3))\cong\pi_3SO(3)\cong\Z$$
is the generator of its second homotopy group. Since $\Diff(\S^1\times\S^2)$ acts on $\SC(\S^1\times\S^2,\xi)$ by push--forward, $\delta$ induces a map
$$\S^2\longrightarrow\SC(\S^1\times\S^2,\xi),\quad (\phi,\varphi)\longmapsto\delta(\phi,\varphi)_*\xi.$$
Denote this map also as $\delta$. Its homotopy class yields a first element $[\delta]\in\pi_2(\SC(\S^1\times\S^2,\xi))$.\\

\noindent The second sphere $S$ is defined directly on the space of contact structures. It is a linear contact sphere in the sense of \cite{CP,GG2}. Consider the following three $1$--forms
$$\alpha_0=zd\theta+xdy-ydx,\quad \alpha_1=xd\theta+ydz-zdy,\quad \alpha_2=yd\theta+zdx-xdz$$
and the standard embedding $(e_0,e_1,e_2):\S^2\longrightarrow\R^3$. Define the map
$$S:\S^2\longrightarrow\SC(\S^1\times\S^2,\xi),\quad (\phi,\varphi)\longmapsto e_0(\phi,\varphi)\alpha_0+e_1(\phi,\varphi)\alpha_1+e_2(\phi,\varphi)\alpha_2.$$
We can now prove the following result.
\begin{proposition}\label{prop:S1S2}
$[S]\neq[\delta]$ in $\pi_2(\SC(\S^1\times\S^2,\xi))$ and the order of $[S]$ is infinite.
\end{proposition}
\begin{proof}
Let us compute the contact homogeneous indices for $[S]$ and $[\delta]$ and the statement of Proposition \ref{prop:S1S2} will follow.

Endow $\R^3$ with the flat metric. The space $\SA_l(\R^3)$ is just the space of cooriented $2$--planes in $\R^3$ and the isomorphism $\SA_l(\R^3)\cong\S^2$ is given by the Gauss map. The image $\varepsilon_*[S]$ is the standard embedding $(e_0,e_1,e_2)$ representing the identity in $\pi_2\S^2$. Hence $\varepsilon_*([S])=1\in\pi_2\S^2$ and the order of $[S]$ is infinite. We now compute $\varepsilon_*[\delta]$.

\noindent The choice of point and frame $(p,\tau)$ is $p=(0;0,0,1)$ and $\tau:T_p(\S^1\times\S^2)\longrightarrow\R^3$ is the oriented frame given by the basis $\langle\partial_\theta,\partial_x,\partial_z\rangle$. In particular $\xi(p)=\xi(0;0,0,1)=\ker d\theta=\langle\partial_x,\partial_y\rangle$. In order to compute the contact homogeneous index we need the images
$$\delta(\phi,\varphi)_*\xi(p)=\langle T_p\delta(\phi,\varphi)\partial_x,T_p\delta(\phi,\varphi)\partial_x\rangle$$
It is readily seen that
$$T_p\delta(\phi,\varphi)=
\left(
\begin{array}{cc}
1 & \vec{0}\\
\vec{a}^t & R^\theta_{\phi,\varphi}
\end{array}
\right)$$
for certain vector $\vec{a}$ depending on $(\theta;\phi,\varphi)$. In
particular at $p=(0;0,0,1)$, this yields $R^0_{\phi,\varphi}=Id$ and thus $T_p\delta(\phi,\varphi)\partial_x=\partial_x$ and $T_p\delta(\phi,\varphi)\partial_y=\partial_y$. The image $\delta(\phi,\varphi)_*\xi(p)$ is constant and we conclude $$\varepsilon_*([\delta])=0\in\pi_2\S^2.$$
The contact homogeneous indices of $[S]$ and $[\delta]$ differ and thus both spheres are not homotopic.\end{proof}
Proposition \ref{prop:S1S2},  combined with the construction described in \cite{CP}, has a simple corollary.
\begin{corollary}\label{cor:S1S2}
There exists a positive non--contractible loop in $\pi_1\Cont(\S^1\times\S^2,\xi)$.
\end{corollary}
\begin{proof}
The image of $[S]\in\pi_2\SC(\S^1\times\S^2,\xi)$ via the connecting morphism $\partial$ in the exact sequence
$$\ldots\longrightarrow\pi_2\Diff(\S^1\times\S^2)\stackrel{i}{\longrightarrow}\pi_2\SC(\S^1\times\S^2,\xi)\stackrel{\partial}{\longrightarrow}\pi_1\Cont(\S^1\times\S^2,\xi)\longrightarrow\ldots$$
is an element $\Gamma\in\pi_1\Cont(\S^1\times\S^2,\xi)$. The construction of the loop at infinity provides a positive representative
$$\gamma:\S^1\longrightarrow\Cont(\S^1\times\S^2,\xi),\quad [\gamma]=\Gamma.$$
We refer the reader to \cite{CP} for an explicit construction. Since $\pi_2\Diff(\S^1\times\S^2)\cong\Z$ is generated by the class $[\delta]$, Proposition \ref{prop:S1S2} implies that $[S]\not\in\im i=\ker\partial$. Hence $\gamma$ is a positive non--contractible loop.
\end{proof}
In other words, the group $\Cont(\S^1\times\S^2,\xi)$ is non--orderable, see \cite{CP2, EKP,Gi,GD} for a detailed account. The statement of the Corollary also follows from Proposition 2.1.B in \cite{EP} applied to the loop
$$f:\S^1\longrightarrow \Cont(\S^1\times\S^2,\xi),\quad \sigma\longrightarrow f(\sigma)(\theta;x,y,z)=(\theta;R^\sigma_{\pi,0}(x,y,z)).$$
with non--negative Hamiltonian $H(\theta;x,y,z)=x^2+y^2\geq0$. The construction in Corollary \ref{cor:S1S2} has the advantage of being explicit.


\end{document}